\documentclass[10pt]{article}
\usepackage{amsmath,amssymb,amsbsy,amsfonts,amsthm,latexsym,
            amsopn,amstext,amsxtra,euscript,amscd,amsthm,url}
\usepackage[usenames, dvipsnames]{color}
\usepackage[utf8]{inputenc}
\usepackage[english]{babel}
\usepackage{mathrsfs}
\usepackage{caption}
\usepackage{tikz}
\newtheorem{thm}{Theorem}

\newtheorem{lem}{Lemma}

\newcommand{\Ocal}{O}

\global\long\def\epsilon{\varepsilon} 
\usepackage{amsmath}
\DeclareMathAlphabet{\mathpzc}{OT1}{pzc}{m}{it}

\begin{document}
\date{}

\title{\bf On the weighted average number of subgroups of ${\mathbb {Z}}_{m}\times {\mathbb {Z}}_{n}$
with $mn\leq x$}   
\author{Isao Kiuchi and Sumaia Saad Eddin}

\maketitle
{\def\thefootnote{}
\footnote{{\it Mathematics Subject Classification 2010: 11A25, 11N37, 11Y60.\\ 
Keywords: Number of subgroups; Number of cyclic subgroups; Dirichlet series; Divisor function.}}

\begin{abstract}
Let $\mathbb{Z}_{m}$ be the additive group of residue classes modulo $m$. For any positive integers $m$ and $n$, let  $s(m,n)$ and $c(m,n)$
denote the total number of subgroups and cyclic subgroups of the group ${\mathbb{Z}}_{m}\times {\mathbb{Z}}_{n}$, respectively. 
Define
$$ 
\widetilde{D}_{s}(x) = \sum_{mn\leq x}s(m,n)\log\frac{x}{mn}
\ \ \ \ {\rm{and}} \ \ \  \
\widetilde{D}_{c}(x) = \sum_{mn\leq x}c(m,n)\log\frac{x}{mn}. 
$$ 
In this paper, we study the asymptotic behaviour of functions 
$\widetilde{D}_{s}(x)$ and  $\widetilde{D}_{c}(x)$. 
\end{abstract}

\section{Introduction and main result}


Let $\mathbb{Z}_{m}$ be the additive group of residue classes modulo $m$. Let $\mu$, $\tau$ and $\phi$ be the M\"{o}bius function, the divisor function and the Euler totient function, respectively. For any positive integers $m$ and $n$,    
$s(m,n)$ and $c(m,n)$
denote the total number of subgroups and cyclic subgroups of ${\mathbb{Z}}_{m}\times {\mathbb{Z}}_{n}$, respectively. The properties of the subgroups of the group ${\mathbb{Z}}_{m}\times {\mathbb{Z}}_{n}$ were studied by Hampejs, Holighaus, T\'{o}th and Wiesmeyr in \cite{HHTW}. We recall that $\gcd(m,n)$ is the greatest common divisor  of $m$ and $n$. The authors deduced formulas for $s(m, n)$ and $c(m, n)$, using a simple elementary method. They showed that
\begin{align*}                                                                  
s(m,n) &= \sum_{a|m, b|n}\gcd(a, b)= \sum_{d|\gcd(m,n)}\phi(d)\tau \left(\frac{m}{d}\right) \tau\left(\frac{n}{d}\right)\\
&=
\sum_{d|\gcd(m,n)}d\tau \left(\frac{mn}{d^2}\right),
\end{align*}
and 
\begin{align*}                     
c(m,n) &= \sum_{a|m, b|n \atop \gcd(m/a, n/b)=1}\gcd(a, b)=\sum_{a|m, b|n}\phi(\gcd(a, b))\\
&= \sum_{d|\gcd(m,n)}(\mu*\phi)(d)\tau\left(\frac{m}{d}\right)\tau\left(\frac{n}{d}\right)= \sum_{d|\gcd(m,n)}\phi(d)\tau\left(\frac{mn}{d^2}\right).
\end{align*}
Here, as usual, the symbol $*$ denotes the Dirichlet convolution of two arithmetical functions $f$ and $g$ defined by $(f * g)(n) = \sum_{d \mid n} f(d) g(n / d)$, for every positive integer $n$.
Suppose $x>0$ is a real number. Define
$$
S^{(1)}(x) :=\sum_{m,n\leq x}s(m,n), \quad 
S^{(2)}(x) :=\sum_{m,n\leq x \atop \gcd(m, n)>1}s(m,n) 
$$

$$
S^{(3)}(x) :=\sum_{m,n\leq x}c(m,n), \quad 
S^{(4)}(x) :=\sum_{m,n\leq x \atop \gcd(m, n)>1}c(m,n) 
$$
The functions $S^{(2)}(x)$ and $S^{(4)}(x)$ represent the number of total subgroups, and cyclic subgroups of ${\mathbb{Z}}_{m}\times {\mathbb{Z}}_{n}$, respectively, having rank two, with $m, n\leq x$. 
W.G. Nowak and L. T\'{o}th~\cite{NT} studied the above functions and proved that   
\begin{align*}                                                                 
S^{(j)}(x)=x^{2}\left(\sum_{r=0}^{3}A_{j,r}(\log x)^{r}\right) + O\left(x^{\frac{1117}{701}+\varepsilon}\right), 
\end{align*}
where $A_{j,r}\ (1\leq j \leq 4,\ 0\leq r \leq 3)$ are computable constants. Moreover, they showed that the double Dirichlet series of the functions $s(m,n)$ and $c(m,n)$ can be represented by the Riemann zeta function. 
Later, the above error term has been improved by T\'{o}th and  Zhai \cite{TZ} to  
$
O\left(x^{\frac{3}{2}}(\log x)^{\frac{13}{2}}\right). 
$

More recently, Sui and Liu \cite{SL} considered the sum of $s(m,n)$ and of $c(m,n)$ in the Dirichlet region $\{(m,n):\ m, n\leq x\}$. Define
$$
D_{s}(x) := \sum_{mn\leq x}s(m,n)
\ \ \ {\rm{and}} \ \ \  
D_{c}(x) := \sum_{mn\leq x}c(m,n),
$$
the authors obtained two asymptotic formulas of $D_{s}(x) $ and $D_{c}(x)$ by using the method of exponential sums. They proved that 
\begin{align*}                          
D_{s}(x) = xP_{4}(\log x) + O\left(x^{2/3}(\log x)^6 \right)
\end{align*}
and 
\begin{align*}                      
D_{c}(x) = xR_{4}(\log x) + O\left(x^{2/3}(\log x)^6 \right),
\end{align*}
where $P_{4}(u)$ and  $R_{4}(u)$ are polynomials in $u$ of degree $4$ with the leading coefficients $1/(8\pi^2)$ and $3/(4\pi^4)$, respectively. Put
$$
\Delta_s(x):=D_{s}(x) - xP_{4}(\log x), \quad 
\Delta_c(x):=D_{c}(x) -xR_{4}(\log x),
$$
Sui and Liu also studied the upper bound of the mean-square estimate of $\Delta_s(x)$ and $\Delta_c(x)$ and guessed that $\Delta_s(x), \Delta_c(x)\ll x^{41/72+\epsilon}$ hold on average. 
Moreover, they conjectured that 
$\Delta_s(x), \Delta_c(x)\ll x^{1/2+\epsilon}.$\\

In this paper, we study the weighted average of $s(m,n)$  and $c(m,n)$ with weight concerning logarithms. Let   
$$ 
\widetilde{D}_{s}(x) = \sum_{mn\leq x}s(m,n)\log\frac{x}{mn}
\ \ \ \ {\rm{and}} \ \ \  \
\widetilde{D}_{c}(x) = \sum_{mn\leq x}c(m,n)\log\frac{x}{mn},
$$ 
then, we have the following results. 
 \begin{thm}      
 \label{th1}
Let the notation be as above.  
For any positive real number $x>2$, we have                                
\begin{align}                          \label{KS-11}
\widetilde{D}_{s}(x) = x\widetilde{P}_{4}(\log x) + O\left(x^{\frac{1}{2}}\log^{}x\right),
\end{align}
and
\begin{align}                           \label{KS-21}
\widetilde{D}_{c}(x) = x\widetilde{R}_{4}(\log x) + O\left(x^{\frac{1}{2}}\log^{}x\right),
\end{align}
 where $\widetilde{P}_{4}(u)$ and  $\widetilde{R}_{4}(u)$ are polynomials in $u$ of degree $4$ with computable coefficients.    
\end{thm}

\section{Auxiliary results}

In order to prove our main result, we first show some necessary lemmas.

\begin{lem}                                                            
\label{lem10}
For every $z, w \in \mathbb{C}$ with $\Re({z}), \Re({w}) >1$, we have  
\begin{align}                                                               \label{lem11}
\sum_{m,n=1}^{\infty}\frac{s(m,n)}{m^{z}n^{w}} 
        = \frac{\zeta^{2}(z)\zeta^{2}(w)\zeta(z+w-1)}{\zeta(z+w)},    
\end{align}
and  
\begin{align}                          \label{lem12}
\sum_{m,n=1}^{\infty}\frac{c(m,n)}{m^{z}n^{w}} 
        = \frac{\zeta^{2}(z)\zeta^{2}(w)\zeta(z+w-1)}{\zeta^{2}(z+w)}. 
\end{align}
\end{lem}
\begin{proof} 
The proof can be found in \cite[Theorem~1]{NT}. 
\end{proof}
\begin{lem}               
\label{lem21}
For $t\geq t_0>0$ uniformly in $\sigma$, we have 
\begin{align*}
\zeta(\sigma+it)
&\ll \left\{\begin{array}{cl}
t^{\frac16 (3-4\sigma)}\log t & \left(0\leq \sigma \leq \frac12\right), \ \smallskip \\
t^{\frac13 (1-\sigma)} \log t & \left(\frac12 \leq \sigma \leq 1\right).
\end{array} \right.
\end{align*}
Moreover, for $\sigma>1$ we have  
\begin{align*}
&\zeta(\sigma+it)
\ll \min \left( \frac{1}{\sigma-1},\  \log (|t|+2)\right)\\ 
&
\zeta^{-1}(\sigma+it)
\ll \min \left( \frac{1}{\sigma-1}, \ \log (|t|+2)\right).
\end{align*}
\end{lem}
\begin{proof}
The first estimate follows immediately from \cite[Theorem II.3.8]{Te}.
The second and third estimates can be found in \cite{SL}.
\end{proof}
\begin{lem}                                          \label{lem40}
We have  
\begin{equation}                                                               \label{JS-1}
\int_{1}^{T}\frac{|\zeta(1/2+it)|^4}{|\zeta(1+2it)|}\, dt \ll T(\log T)^4. 
\end{equation} 
\end{lem}
\begin{proof} 
The proof of this result can be deduced from \cite[ Proposition 2]{JS} when $k=1$.  
\end{proof}
\section{Proof of Theorem~\ref{th1}} 
Our proof is similar in spirit to the proof of Theorem 1.1 in \cite{SL}. Both proofs are based on the residue theorem and the classical method to estimate the integrals. Here, we only prove our theorem for the function $\widetilde{D}_{s}(x)$. The proof of the function $ \widetilde{D}_{c}(x)$ is similar. Suppose that the double Dirichlet series of $s(m,n)$ 
$$
\alpha(s)=\sum_{m,n=1}^{\infty}\frac{s(m,n)}{(mn)^s}, 
$$
has abscissa of convergence $\sigma_c$.
Applying Riesz typical means, (see \cite[Chapter 5: (5.21) and (5.22)]{MV}), show that 
$$\alpha(s)=s^2\int_{1}^{\infty}\widetilde{D}_{s}(x)x^{-s-1}\, dx
$$
and that 
$$
\widetilde{D}_{s}(x)=\frac{1}{2\pi i}\int_{\sigma_0-i\infty}^{\sigma_0+i\infty}\alpha(s)\frac{x^s}{s^2}\, ds,
$$

when $x>0$ and $\sigma_0>\max(0, \sigma_c).$
Using Lemma \ref{lem10} with $\sigma_0=1+1/\log x$, we have
$$
\widetilde{D}_{s}(x):=\sum_{m n\leq x} s(m,n)\log \frac{x}{mn}=\frac{1}{2\pi i}\int_{\sigma_0-i\infty}^{\sigma_0+i\infty}\frac{\zeta^{4}(s)\zeta(2s-1)}{\zeta(2s)}\, \frac{x^{s}}{s^2}\, ds.
$$
Let $T\geq 1$ be a larger parameter, then the above equality can be rewritten as follows
$$
\widetilde{D}_{s}(x)=\frac{1}{2\pi i}\int_{\sigma_0-iT}^{\sigma_0+iT}\frac{\zeta^{4}(s)\zeta(2s-1)}{\zeta(2s)}\, \frac{x^{s}}{s^2}\, ds +\Ocal\left( \frac{x^{1+\epsilon}}{T}\right),
$$
where we used Lemma~\ref{lem21} to estimate 
\begin{align*}
\frac{1}{2\pi i}\int_{\sigma_{0}\pm iT}^{\sigma_{0} \pm i\infty}  \frac{\zeta^{4}(s)\zeta(2s-1)}{\zeta(2s)}\, \frac{x^{s}}{s^2}\, ds
&\ll \int_{T}^{+\infty}\frac{\left|\zeta^{4}(\sigma_{0} \pm it)\right|\, \left|\zeta(2\sigma_{0}-1 \pm i2t)\right|}{|\zeta(2\sigma_{0} \pm i2t)|}\, \frac{x^{\sigma_{0}}}{t^2}dt    \\
&\ll x^{\sigma_{0}+\varepsilon}\int_{T}^{+\infty} \frac{dt}{t^2} 
\ll  \frac{x^{\sigma_{0}+\varepsilon}}{T}. 
\end{align*}
We consider a rectangle $D$ in the $s$ plane with vertices at the points $1/2-iT$, $\sigma_0-iT$, $\sigma_0+iT$ and $1/2+iT$, where $T\geq 1$. Notice that our function 
$$\frac{\zeta^{4}(s)\zeta(2s-1)}{\zeta(2s)}\, \frac{x^{s}}{s^2}
$$
has a pole at $s=1$ of order $5.$ Then, by the residue theorem we find that
\begin{align}                                                     \label{SL-integral}
R(x,T)= I_1(x,T) - I_{2}(x,T) - I_{3}(x,T) + I_{4}(x,T),  
\end{align}
where   
\begin{align*}
&R(x,T) = \underset{s=1}{\rm{Res}}~\frac{\zeta^{4}(s)\zeta(2s-1)}{\zeta(2s)}\, \frac{x^{s}}{s^2},
\\ &
I_1(x,T) = \frac{1}{2\pi i}\int_{\sigma_0-iT}^{\sigma_0+iT}\frac{\zeta^{4}(s)\zeta(2s-1)}{\zeta(2s)}\, \frac{x^{s}}{s^2}\, ds, 
\\ &
I_{2}(x,T) = \frac{1}{2\pi i}\int_{\frac{1}{2}+iT}^{\sigma_0+iT}\frac{\zeta^{4}(s)\zeta(2s-1)}{\zeta(2s)}\, \frac{x^{s}}{s^2}\, ds, 
\\ &
I_{3}(x,T) = \frac{1}{2\pi i}\int_{\frac{1}{2}-iT}^{\frac{1}{2}+iT}\frac{\zeta^{4}(s)\zeta(2s-1)}{\zeta(2s)}\, \frac{x^{s}}{s^2}\, ds, 
\\ &
I_{4}(x,T) = \frac{1}{2\pi i}\int_{\frac{1}{2}-iT}^{\sigma_0-iT}\frac{\zeta^{4}(s)\zeta(2s-1)}{\zeta(2s)}\, \frac{x^{s}}{s^2}\, ds.  
\end{align*}
For calculating $R(x, T)$, we recall that the Laurent series expansion of the Riemann zeta function at $s=1$ is given by 
$$
\zeta(s)=\frac{1}{s-1}+\gamma_0+\gamma_1(s-1)+\gamma_2(s-1)^2+\gamma_3(s-1)^3+\cdots,
$$
where the constants $\gamma_n$ are often called the Stieltjes constants or generalized Euler constants (see \cite{I}). In particular, $\gamma_0 = 0.57721 56649\cdots$ is the well-known the Euler constant. Using the above expansion, the function $f(s)=\zeta^4(s)\zeta(2s-1)$ can be written as 
\begin{multline*}
f(s):=\zeta^4(s)\zeta(2s-1)=\frac{1/2}{(s-1)^5}+\frac{3\gamma_0}{(s-1)^4}+\frac{7\gamma_0^2-4\gamma_1}{(s-1)^3}\\+\frac{8\gamma_0^3-18\gamma_0\gamma_1+3\gamma_2}{(s-1)^2}
+\frac{\frac{9}{2}\gamma_0^4-30\gamma_0^2\gamma_1+11\gamma_1^2+13\gamma_0\gamma_2-\frac{5}{3}\gamma_3}{s-1}\\
+\left( \gamma_0^5-22\gamma_0^3\gamma_1+36\gamma_0\gamma_1^2+21\gamma_0^2\gamma_2-15\gamma_1\gamma_2-7\gamma_0\gamma_3+\frac{3}{4}\gamma_4\right)+
\Ocal(s-1).
\end{multline*}
Let $g(s,x)=x^ss^{-2}\zeta^{-1}(2s)$, then, we have 
\begin{align*}
R(x, T):=\underset{s=1}{\rm{Res}}~f(s)g(s, x)&=\frac{1}{2\times 4!} g^{(4)}(1,x)+\frac{4}{4!}\times 3\gamma_0\times g^{(3)}(1,x)\\&+\frac{12}{4!}\times \left( 7\gamma_0^2-4\gamma_1\right)\times g^{(2)}(1,x)\\&
+\frac{24}{4!}\times \left(8\gamma_0^3 -18\gamma_0\gamma_1+3\gamma_2\right)\times g^{\prime}(1,x) \\ &
+\frac{24}{4!}\times \left(\frac{9}{2}\gamma_0^4 -30\gamma_0^2\gamma_1+11\gamma_1^2+13\gamma_0\gamma_2-\frac{5}{3}\gamma_3\right)\times g(1,x),
\end{align*}
where $g^{(i)}(1,x)$ denotes the $i$-th derivative of the function $g(s,x)$ with respect to $s$ at $s=1$. By careful calculations, we find that 
\begin{align*}
g^{(4)}(1,x)&= \frac{6}{\pi ^2}x(\log x)^4-\frac{48}{\pi^4}(\pi^2+6\zeta^{\prime}(2))x(\log x)^3\\&+\frac{216}{\pi^6}\left(-4\pi^2\zeta^{\prime \prime}(2)+48\zeta^{\prime}(2)^2+8\pi^2\zeta^{\prime}(2)+\pi^4\right)x(\log x)^2\\&
+\frac{576}{\pi^8}\left(-432\zeta^{\prime}(2)^3+72\pi^2\zeta^{\prime}(2)\zeta^{\prime\prime}(2)-72\pi^2\zeta^{\prime}(2)^2-
\pi^6 \right) x\log x\\&
+\frac{576}{\pi^4}\left(-2\zeta^{(3)}(2)-9\zeta^{\prime}(2)+6\zeta^{\prime\prime}(2)\right) x\log x
\\&
+\frac{144}{\pi^{10}}\left(20736\zeta^{\prime}(2)^4+1728\pi^2\zeta^{\prime}(2)^2\left( 2\zeta^{\prime}(2)-3\zeta^{\prime\prime}(2)\right)+5\pi^8\right)x\\&
+\frac{6912}{\pi^{6}}\left(3\zeta^{\prime\prime}(2)^2+9\zeta^{\prime}(2)^2-4\zeta^{\prime}(2)(3\zeta^{\prime\prime}(2)-\zeta^{(3)}(2)) \right)x\\& 
+\frac{576}{\pi^{4}}\left(-\zeta^{(4)}(2)+4\zeta^{(3)}(2)-9\zeta^{\prime\prime}(2)+12\zeta^{\prime}(2) \right)x,
\end{align*}

\begin{align*}
g^{(3)}(1,x) &= \frac{6}{\pi ^2}x(\log x)^3-\frac{36}{\pi^4}\left(6\zeta^{\prime}(2)+\pi^2\right)x(\log x)^2\\&+\frac{108}{\pi^6}\left(48\zeta^{\prime}(2)^2+4\pi^2\left(2\zeta^{\prime}(2)-\zeta^{\prime\prime}(2) \right)+\pi^4\right)x\log x\\&
-\frac{144}{\pi^{8}}\left(432\zeta^{\prime}(2)^3+72\pi^2\zeta^{\prime}(2)\left( \zeta^{\prime}(2)-\zeta^{\prime\prime}(2)\right)+\pi^6 \right)x\\&
-\frac{144}{\pi^4}\left(9\zeta^{\prime}(2)-6\zeta^{\prime\prime}(2)+2\zeta^{(3)}(2) \right)x,
\end{align*}
\begin{align*}
g^{(2)}(1,x)&= \frac{6}{\pi ^2}x(\log x)^2-\frac{24}{\pi^4}\left(6\zeta^{\prime}(2)+\pi^2\right)x\log x
\\&
+\frac{36}{\pi^6}\left( 48\zeta^{\prime}(2)^2+ 4\pi^2\left(2\zeta^{\prime}(2)-\zeta^{\prime\prime}(2)\right)+\pi^4\right) x
\end{align*}
and that
\begin{align*}
g^{\prime}(1,x)= \frac{6}{\pi ^2}x\log x-\frac{12}{\pi^4}(6\zeta^{\prime}(2)+\pi^2) x, \quad \quad g(1, x)=\frac{x}{\zeta(2).}
\end{align*}
Thus, the function $R(x, T)$ becomes 
\begin{align}  
\label{residue}
R(x,T) = x\sum_{r=0}^{4}B_{r}(\log x)^r, 
\end{align}
where 
\begin{align*}
&B_{4}=\frac{1}{8\pi^2},  \\&
B_3=\frac{3\gamma_0-1}{\pi^2}-\frac{6}{\pi^4}\, \zeta^{\prime}(2),\\&
B_2=-\frac{18}{\pi^4}\zeta^{\prime\prime}(2)+\frac{216}{\pi^6}\zeta^{\prime}(2)^2+\frac{36}{\pi^4}(-3\gamma_0+1)\zeta^{\prime}(2)+\frac{3}{2\pi^2}\left(3+14\gamma_0^2-12\gamma_0-8\gamma_1\right),\\&
B_1=-\frac{5184}{\pi^8}\zeta^{\prime}(2)^3+\frac{864}{\pi^6}\zeta^{\prime}(2)\left(-\zeta^{\prime}(2)+3\gamma_0\zeta^{\prime}(2)+\zeta^{\prime\prime}(2) \right)\\&
-\frac{12}{\pi^4}\left(42\gamma_0^2\zeta^{\prime}(2)+\zeta^{\prime}(2)\left(9-24\gamma_1\right)-6\zeta^{\prime\prime}(2)+18\gamma_0\left(-2\zeta^{\prime}(2)+\zeta^{\prime\prime}(2)\right)+2\zeta^{(3)}(2)\right)\\&+\frac{6}{\pi^2}\left(8\gamma_0^3-14\gamma_0^2+9\gamma_0-18\gamma_0\gamma_1+8\gamma_1+3\gamma_2-2\right),\\&
B_0= \frac{1}{\pi^2}\left(15-96\gamma_0^{3}+27\gamma_0^{4}-72\gamma_1+66\gamma_1^{2}-18\gamma_0^2\left( -7+10\gamma_1\right)-36\gamma_2\right)\\& 
+\frac{1}{\pi^2}\left(6\gamma_0\left( -12+36\gamma_1+13\gamma_2\right)-10\gamma_3 \right)
-\frac{5184}{\pi^8}\zeta^{\prime}(2)^2\left( -2\zeta^{\prime}(2)+6\gamma_0\zeta^{\prime}(2)+3\zeta^{\prime\prime}(2)\right)\\&
-\frac{12}{\pi^4}\left( 48\gamma_0^3\zeta^{\prime}(2)+6\left( -2+8\gamma_1+3\gamma_2\right)\zeta^{\prime}(2)+9\zeta^{\prime\prime}(2)-24\gamma_1\zeta^{\prime\prime}(2)+42\gamma_0^2\left( -2\zeta^{\prime}(2)+\zeta^{\prime\prime}(2)\right)\right)\\&
+\frac{12}{\pi^4}\left( 6\gamma_0\left( 9\left( -1+2\gamma_1\right)\zeta^{\prime}(2)+6\zeta^{\prime\prime}(2)-2\zeta^{(3)}(2)\right)+4\zeta^{(3)}(2)-\zeta^{(4)}(2)\right)
\\&
+\frac{144}{\pi^6}\left( 42\gamma_0^2\zeta^{\prime}(2)^2+\left( 9-24\gamma_1\right)\zeta^{\prime}(2)^2-36\gamma_0\zeta^{\prime}(2)\left( \zeta^{\prime}(2)-\zeta^{\prime\prime}(2)\right)+3\zeta^{\prime\prime}(2)^2\right)\\&
+\frac{144}{\pi^6}\left(4\zeta^{\prime}(2)\left(-3\zeta^{\prime\prime}(2)+\zeta^{(3)}(2) \right)\right)+\frac{62208}{\pi^{10}}\zeta^{\prime}(2)^4.
\end{align*}

Again, we use Lemma~\ref{lem21} to estimate the function $I_2(x, T)$  
\begin{align*}
& I_{2}(x,T) =
\frac{1}{2\pi i}\left(\int_{\frac12}^{\frac34}+\int_{\frac34}^{\sigma_0} \right)
\frac{\zeta^{4}(\sigma+iT) \zeta(2\sigma-1+2iT)}{\zeta^{}(2\sigma+2iT) (\sigma+iT)^2} x^{\sigma+iT} d\sigma \\
&\ll 
\frac{\log^{}T}{T^2} \left(\int_{\frac12}^{\frac34}+\int_{\frac34}^{\sigma_0} \right)
|\zeta^{}(\sigma+iT)|^4 |\zeta(2\sigma-1+2iT)| x^{\sigma} d\sigma \\
&\ll 
T^{\frac12}(\log T)^6 \int_{\frac12}^{\frac34} \left(\frac{x}{T^{\frac83}}\right)^{\sigma}d\sigma
+ (\log T)^6 \int_{\frac34}^{\sigma_0} \left(\frac{x}{T^{2}}\right)^{\sigma} d\sigma \\
&\ll 
\left(\frac{x}{T^2} + \left(\frac{x^{}}{T^{2}}\right)^{\frac34}+\left(\frac{x^{}}{T^{\frac53}}\right)^{\frac12}\right)(\log T)^6.
\end{align*}
choosing $T=x$, we deduce that $I_{2}(x,T)\ll x^{-1/3}(\log x)^6. $
A similar argument shows that the function $I_4(x, T)$ is estimated by $x^{-1/3}(\log x)^6.$

Next we estimate $I_{3}(x,T)$. 
Using Lemmas \ref{lem21} and \ref{lem40} we find that
\begin{align*}                                       \label{vertical}
I_{3}(x,T) 
&= \frac{1}{2\pi}\int_{-T}^{T}\frac{\zeta^{4}(\frac12 +it)\zeta(2it)}{\zeta(1+2it)(\frac{1}{2}+it)^2}~x^{\frac12 +it}dt \nonumber \\
&\ll x^{\frac12} + x^{\frac12} \int_{1}^{T} \frac{|\zeta(\frac12 +it)|^4}{|\zeta(1+2it)|}\, \frac{|\zeta(2it)|}{t^2}dt \nonumber \\ 
&\ll x^{\frac12} 
  +  x^{\frac12} \sum_{k\leq\frac{\log T}{\log 2}}
     \int_{2^{k-1}}^{2^k} \frac{|\zeta(\frac12 +it)|^4}{|\zeta(1+2it)|}\, \frac{|\zeta(2it)|}{t^2}dt \nonumber \\ 
&\ll x^{\frac12} 
  +  x^{\frac12} \sum_{k\leq\frac{\log T}{\log 2}}\frac{1}{(2^{k})^{\frac32}}
     \int_{2^{k-1}}^{2^k} \frac{|\zeta(\frac12 +it)|^4}{|\zeta(1+2it)|} dt \nonumber \\
&\ll x^{\frac12} 
  +  x^{\frac12} \sum_{k\leq\frac{\log T}{\log 2}}
       \frac{1}{(2^{k})^{\frac32}}\cdot 2^{k}k^{4}  \nonumber \\
&\ll x^{\frac12}  + x^{\frac12} \sum_{k\leq k_{0}} \frac{k^{4}}{(2^{\frac12})^{k}}  
  +  x^{\frac12} \sum_{k_{0}<k\leq\frac{\log T}{\log 2}} 1    \nonumber  \\    
&\ll  x^{\frac12} \log T. 
\end{align*}
Combining the above results with \eqref{SL-integral}, we get the desired conclusion.
 
\section*{Acknowledgement}
The second author is supported by the Austrian Science Fund (FWF): Projects F5507-N26 and F5505-N26, which are part
of the Special Research Program  ``Quasi Monte
Carlo Methods: Theory and Applications''.

\medskip\noindent {\footnotesize Isao Kiuchi: Department of Mathematical Sciences, Faculty of Science,
Yamaguchi University, Yoshida 1677-1, Yamaguchi 753-8512, Japan. \\
e-mail: {\tt kiuchi@yamaguchi-u.ac.jp}}

\medskip\noindent {\footnotesize Sumaia Saad Eddin: 
Institute of Financial Mathematics and Applied Number Theory, Johannes Kepler University, Altenbergerstrasse 69, 4040 Linz, Austria.\\
e-mail: {\tt sumaia.saad\_eddin@jku.at}}

\end{document}